%% file: main.tex
\pdfoutput=1

\documentclass[final]{amsart}
\input{settings}
\begin{document}

\title{On the Isoperimetric Profile of the Hypercube}

\author{Federico Glaudo}
\address{Federico Glaudo
\hfill\break School of Mathematics, Institute for Advanced Study, 1 Einstein Dr., Princeton NJ 05840, U.S.A.}
\email{fglaudo@ias.edu}

\begin{abstract}
We prove that a subset of the hypercube $(0,1)^d$ with volume sufficiently close to $\frac12$ has (relative) perimeter greater than or equal to $1$. This settles a conjecture by Brezis and Bruckstein.
We also prove that, in contrast with what happens for the high-dimensional sphere $\mathbb S^d$, the isoperimetric profile of the hypercube $(0,1)^d$ does not converge to the Gaussian isoperimetric profile as $d\to\infty$. 
\end{abstract}

\maketitle

\section{Introduction}

For an open set $\Omega\subseteq\R^d$, the relative isoperimetric problem in $\Omega$ consists of minimizing the perimeter (in $\Omega$) of a set $E\subseteq\Omega$ with fixed volume. So, given $0<\lambda<\abs{\Omega}$, one is interested in the minimization problem
\begin{equation}\label{eq:rel-isop}
    I_\Omega(\lambda)\defeq\inf \big\{\Per(E, \Omega): \, E\subseteq \Omega \text{ so that } \abs{E}=\lambda\big\},
\end{equation}
where $\Per(E, \Omega)$ denotes the perimeter of $E$ inside $\Omega$; if $E$ has a smooth boundary then $\Per(E, \Omega)$ coincides with $\Haus^{d-1}(\Omega\cap\partial E)$ (see \cref{subsec:perimeter} for the general definition).
The (relative) isoperimetric profile of $\Omega$, denoted by $I_{\Omega}:[0,\abs{\Omega}]\to [0,\infty)$ is the function so that $I_{\Omega}(\lambda)$ is the value of the infimum appearing in \cref{eq:rel-isop} (and $I_{\Omega}(0)=0$ and $I_\Omega(\abs{\Omega})=0$ if $\abs{\Omega}<\infty$).

The (relative) isoperimetric problem is a classical question with a multitude of applications that has received considerable attention recently. The vast literature on the topic makes it hard to give a complete list of references, so we refer the reader to the three recent works \cite{LeonardiRitoreVernadakis2022,FuscoMorini2023,AntonelliBrueFogagnoloPozzetta2022} and to the references therein.
This paper investigates in particular the isoperimetric profile of the $d$-dimensional hypercube $(0,1)^d$.

\subsection{The relative isoperimetric problem in \texorpdfstring{$(0,1)^d$}{the hypercube}}
To frame appropriately our results, let us recall what is known about the relative isoperimetric problem in the hypercube.

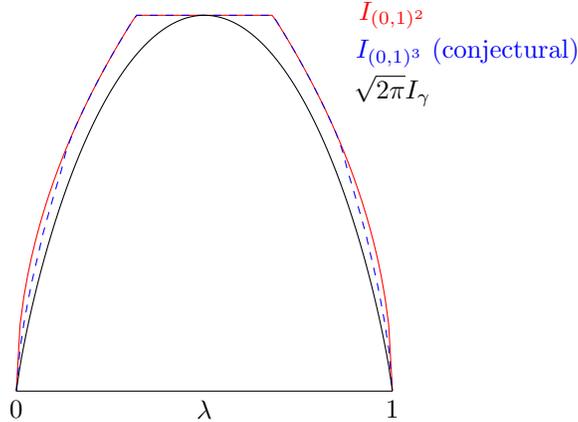
\begin{figure}[htb]
\centering
\input{isop_profiles.tikz}
\caption{The blue dashed graph represents the conjectural profile of the cube $(0,1)^3$. The figure shows a number of features of the problem: the isoperimetric profile is concave, the lower bound $\sqrt{2\pi}I_\gamma$ is remarkably close to the actual value of $I_{(0,1)^d}$, the profile $I_{(0,1)^d}$ is constant in a neighborhood of $\lambda=\frac12$, the profiles $I_{(0,1)^d}$ are decreasing with respect to the dimension $d\ge 1$. }\label{fig:isop_profiles}
\end{figure}

In the case of the square $(0,1)^2$, the isoperimetric profile (along with the minimizers of the relative isoperimetric problem) is known~\cite{BrezisBruckstein2021}. In dimension $d=3$ (so, for the cube $(0,1)^3$), it is conjectured~\cite[pg. 11]{Ros2005} (see also \cite[Theorem 9]{Ritore1997}) that the only minimizers are balls, cylinders, and half-spaces intersected with the cube; under this assumption one can determine exactly the isoperimetric profile of $(0,1)^3$. 
In any dimension, since $(0,1)^d$ is a polytope, for small volumes the minimizers of the relative isoperimetric problem are balls centered at the vertices of $(0,1)^d$ (see \cite[Theorem 6.8]{RitoreVernadakis2015}). As an immediate consequence, one gets
\begin{equation*}
    I_{(0,1)^d}(\lambda) = \frac12 d\abs{B_1^{\R^d}}^{\frac1d}\lambda^{\frac{d-1}{d}} \text{ \,\,for $0<\lambda<\lambda_0(d)$,}
\end{equation*}
where $B_1^{\R^d}$ denotes the unit ball in $\R^d$ and $\lambda_0(d)$ is a dimensional constant that goes to $0$ as $d\to\infty$.

In every dimension $d\ge 1$, it was proven by Hadwiger~\cite{Hadwiger1972} that $I_{(0,1)^d}(\tfrac12)=1$, or equivalently that if $E\subseteq(0,1)^d$ has measure $\abs{E}=\tfrac12$ then its perimeter is at least $1$ (i.e., splitting the cube with a hyperplane parallel to one of its faces is optimal).

We show that the same result holds also if the set $E$ has measure sufficiently close to $\tfrac12$. This settles a conjecture by Brezis and Bruckstein~\cite[Open Problem 10.1]{Brezis2023} (see also \cite[Remark 2]{BrezisBruckstein2021}). The result is new already for $d=3$.
\begin{theorem}\label{thm:main-halfmeasure}
    For $d\ge 1$, there exists $\eps_d>0$ so that $I_{(0,1)^d}(\lambda)=1$ for all $\lambda\in (\tfrac12-\eps_d, \tfrac12-\eps_d)$. Equivalently, any set of finite perimeter $E\subseteq (0,1)^d$ with 
    $\abs*{\abs{E} - \frac12} \le \eps_d$ satisfies $\Per(E, (0, 1)^d) \ge 1$.
    Moreover, $\Per(E, (0, 1)^d)=1$ if and only if $E=\{x\in (0,1)^d:\, \hat v\cdot x < \abs{E}\}$ for some $\hat v\in\{\pm e_1, \pm e_2, \dots, \pm e_d\}$.
\end{theorem}

For values of the volume distinct from $\tfrac12$, the exact value of $I_{(0,1)^d}$ is not known, but a remarkable lower bound with the Gaussian isoperimetric profile was established in \cite[Theorem 7]{BartheMaurey2000} (see also \cite[Theorem 7]{Ros2005}, \cite[(2.2)]{AmbrosioBourgainBrezisFigalli2016}).

\begin{theorem*}[{\cite[Theorem 7]{BartheMaurey2000}}]
    Let $I_{\gamma}\defeq\varphi\circ\Phi^{-1}$ be the Gaussian isoperimetric profile (see \cref{subsec:perimeter}), where $\varphi(t)\defeq \frac1{\sqrt{2\pi}}\exp(-\tfrac12t^2)$ and $\Phi(t)\defeq \int_{\infty}^t\varphi(s)\,ds$.

    For any $d\ge 1$, it holds that $I_{(0,1)^d} \ge \sqrt{2\pi} I_{\gamma}$; equivalently
    \begin{equation}\label{eq:tmp994}
        \Per(E, (0,1)^d) \ge \sqrt{2\pi}I_\gamma(\abs{E})
    \end{equation}
    for any set $E\subseteq (0,1)^d$ of finite perimeter.
\end{theorem*}
The lower bound shown in this theorem is remarkably precise already in dimensions $d=2,3$ (see \cref{fig:isop_profiles}) and its precision can only improve in higher dimension as $I_{(0,1)^d}$ is decreasing with respect to the dimension $d$. Furthermore, if instead of the cube $(0,1)^d$, one considers the case of the sphere $\mathbb S^d$ (i.e., one studies the isoperimetric problem in the Riemannian manifold $\mathbb S^d$), it turns out that its isoperimetric profile $I_{\mathbb S^d}$, appropriately rescaled, converges to $I_\gamma$ as the dimension $d\to\infty$ (see \cite[Theorem 10, Proposition 11]{Barthe2001} or \cite[Theorem 21]{Ros2005}).

The facts mentioned in the previous paragraph may lead one to expect that, as the dimension $d\to\infty$, the isoperimetric profile of the cube $I_{(0,1)^d}$ converges to $\sqrt{2\pi}I_\gamma$. This is true when evaluating it at $\lambda \in \{0, \tfrac12, 1\}$.
Unexpectedly for the author, we show that this claim is false, i.e., that there is a gap between $\inf_{d\ge 1} I_{(0,1)^d}$ and $\sqrt{2\pi}I_\gamma$.
\begin{theorem}\label{thm:main-highdim}
    For all $d\ge 1$, we have $I_{(0,1)^{d+1}}\le I_{(0,1)^d}$; let $I_{(0,1)^{\infty}}=\inf_{d\ge 1} I_{(0,1)^d}$.
    The function $I_{(0,1)^\infty}:[0,1]\to[0,1]$ is a concave function such that
    \begin{equation*}
        I_{(0,1)^\infty}(\lambda) > \sqrt{2\pi} I_\gamma(\lambda) \text{ for all $\lambda\in (0,1)\setminus\{\tfrac12\}$}.
    \end{equation*}
\end{theorem}
The proof of \cref{thm:main-highdim} is quantitative (i.e., no compactness is used) and thus one could keep track of all the constants and dependences on $\lambda$ and find an explicit function $g:[0,1]\to[0,\infty)$ (strictly positive on $(0,1)\setminus\{\tfrac12\}$ with $0=g(0)=g(\tfrac12)=g(1)$) such that 
\begin{equation*}
    I_{(0,1)^d}(\lambda) \ge \sqrt{2\pi}I_{\gamma}(\lambda) + g(\lambda) \text{ for all $0\le \lambda\le 1$.}
\end{equation*}
We decided not to do this because it would make the proof more cumbersome and the resulting function $g$ would not be optimal in any sense.

Let us remark that \cref{thm:main-highdim} may also be interpreted as a dimension-free stability result for the isoperimetric inequality~\cref{eq:tmp994}.

\subsection{Open questions}
The results of this paper naturally raise some further questions that we collect here.

\begin{question}
    Does the statement of \cref{thm:main-halfmeasure} hold also with an $\eps$ independent of the dimension? Equivalently, is there an $\eps>0$ so that $I_{(0,1)^d}(\lambda)=1$ for all $\lambda\in[\tfrac12-\eps, \tfrac12+\eps]$ and for all $d\ge 1$?
\end{question}

\begin{question}
    Is it true that for any $\lambda\in [0,1]$, the sequence $(I_{(0,1)^d}(\lambda))_{d\ge 1}$ is eventually constant? Equivalently, for each $\lambda$, does it hold that $I_{(0,1)^d}(\lambda) = I_{(0,1)^\infty}(\lambda)$ for all $d$ sufficiently large (see \cref{thm:main-highdim} for the definition of $I_{(0,1)^{\infty}}$)?
\end{question}

\begin{question}
    Is there an explicit formula for the limiting isoperimetric profile $I_{(0,1)^\infty}$?
\end{question}

\subsection{Methods and organization of the paper}
The foundation of the proofs of the two main results of this paper (namely \cref{thm:main-halfmeasure,thm:main-highdim}) is the rigidity of the inequality~\cref{eq:tmp994}, i.e., if $\Per(E, (0,1)^d) = \sqrt{2\pi} I_\gamma(\abs{E})$ then $E$ is a half-cube. We state and prove this rigidity in \cref{sec:gaussian-cube-rigidity}.

Then, in \cref{sec:proof-halfmeasure} we prove \cref{thm:main-halfmeasure}. The proof is by compactness and uses crucially a recent result~\cite{EdelenLi2022} about the convergence of free boundary minimal surfaces in the case of non-smooth convex domains.

Finally, in \cref{sec:proof-highdim} we show \cref{thm:main-highdim}. The main idea is to reduce the relative isoperimetric problem in $(0,1)^d$ to the following penalized isoperimetric problem in the Gaussian space $(\R^d,\gamma_d)$:
\begin{equation*}
    \inf_{F\subseteq\R^d:\, \gamma_d(F)=\lambda} \Per_{\gamma_d}(F) + \int_{\partial F}\sqrt{\sum_{i=1}^d (\nu_{\partial F})_i^2\exp(x_i^2)} - 1\,d\Haus^{d-1}_{\gamma_d}(x),
\end{equation*}
where $\nu_{\partial F}$ denotes the unit normal to the boundary of $F$ (see \cref{subsec:perimeter} for the definitions of $\gamma_d, \Haus^{d-1}_{\gamma_d}, \Per_{\gamma_d}$).
Then, by using the dimension-free stability of the Gaussian isoperimetric inequality (see \cite{MosselNeeman2015,Eldan2015,BarchiesiBrancoliniJulin2017}), we prove that the two terms of the penalized problem cannot be simultaneously minimized unless $\lambda\in\{0,\tfrac12,1\}$. The result follows.

Let us remark that (even though our proof does not employ this perspective) the statement of \cref{thm:main-halfmeasure} can be interpreted as the fact that for volumes close to $\tfrac12$ the penalized problem is solved by affine half-spaces. This is not the first instance of penalization of the Gaussian isoperimetric problem that preserves the optimality of half-spaces~\cite{BarchiesiJulin2020}.

\subsection*{Acknowledgements}
The author is thankful to S. Aryan for mentioning the problem to him, and to G. Antonelli for many fruitful discussions about the problem and for suggesting various references.
The author is supported by the National Science Foundation under Grant No. DMS--1926686.

\section{Notation and preliminaries}
\subsection{Gaussian Measure}
Let $\varphi:\R\cup\{\pm\infty\}\to (0,\tfrac{1}{\sqrt{2\pi}}]$ be the function $\varphi(t)\defeq \frac1{\sqrt{2\pi}}\exp(-\tfrac12 t^2)$. Denote with $\gamma\in\prob(\R)$ the Gaussian (probability) measure on $\R$, i.e., the measure with density $\varphi$.

Let us define the $d$-dimensional versions of $\varphi$ and $\gamma$ as follows. For any $d\ge 1$, let $\varphi_d:\R^d\to(0,(2\pi)^{-d/2}]$ be
\begin{equation*}
    \varphi_d(x) \defeq \frac1{(2\pi)^{d/2}}\exp\big( -\tfrac12 \abs{x}^2\big)
    = \varphi(x_1)\varphi(x_2)\cdots \varphi(x_d).
\end{equation*}
Let $\gamma_d\in\prob(\R^d)$ be the $d$-dimensional Gaussian (probability) measure, i.e., the measure with density $\varphi_d$ or equivalently $\gamma_d = \gamma\otimes \gamma \otimes\cdots\otimes\gamma$ where we are taking the product of $d$ copies of $\gamma$.

\subsection{Hausdorff Measure and Perimeter}\label{subsec:perimeter}
We denote with $\abs{\emptyparam}$ the Lebesgue measure in the Euclidean space (of any dimension).
We denote with $\Haus^k$ the $k$-dimensional Hausdorff measure in the Euclidean space (of any dimension).

For a set $E\subseteq\R^d$ of finite perimeter (for the theory of sets of finite perimeter we suggest the reader to consult \cite{Maggi2012}), we denote with $\partial^*E$ its reduced boundary~\cite[Chapter 15]{Maggi2012} (which coincides with the topological boundary if $E$ is sufficiently regular). Let us recall that the reduced boundary is a $(d-1)$-rectifiable set and thus admits a normal vector $\Haus^{d-1}$-almost everywhere. The perimeter of $E$ in an open set $\Omega$ is defined as\footnote{Since $E$ is a set of finite perimeter, its indicator function $\mathds{1}_E$ is a function of bounded variation and thus its distributional derivative is a measure.}
\begin{equation*}
    \Per(E, \Omega) \defeq \norm{D{\mathds 1}_E}(\Omega) = \Haus^{d-1}(\partial^*E\cap\Omega).
\end{equation*}

Let us now give the analogous definitions in the Gaussian setting. 
Let us denote with $\Haus^{k}_{\gamma_d}\defeq \varphi_d\Haus^k$ the $k$-dimensional Hausdorff measure in $\R^d$ weighted by $\varphi_d$.
For $E\subseteq\R^d$ a set of \emph{locally} finite perimeter, its Gaussian perimeter is defined as
\begin{equation*}
    \Per_{\gamma_d}(E) = \Haus^{d-1}_{\gamma_d}(\partial^*E) = \int_{\partial^*E} \varphi_d\,d\Haus^{d-1}.
\end{equation*}

\subsection{The Gaussian Isoperimetric Inequality}\label{subsec:gaussian-isop}
Let $\Phi:\R\cup\{\pm\infty\}\to[0,1]$ be the function $\Phi(t)\defeq \int_{-\infty}^t \varphi(s)\,ds = \gamma((\infty, t))$ and let $I_{\gamma}:[0,1]\to[0,1]$ be $I_\gamma = \varphi\circ \Phi^{-1}$. The function $I_\gamma$ is the isoperimetric profile for the Gaussian space in any dimension, that is, for any positive integer $d\ge 1$ and for any set $E\subseteq\R^d$ of finite perimeter, we have (see \cite{SudakovCirelson1974,Borell1975}, and \cite{Bobkov1997,CarlenKerce2001} for the equality cases)
\begin{equation*}
    \Per_{\gamma_d}(E) \ge I_\gamma(\gamma_d(E))
\end{equation*}
and the equality holds if and only if $E$ is an affine half-space.

\section{Rigidity of half-cubes}\label{sec:gaussian-cube-rigidity}
In this section we study the equality cases of \cref{eq:tmp994}. We show that if a set $E\subseteq (0,1)^d$ satisfies $\Per(E, (0,1)^d)=\sqrt{2\pi}I_\gamma(\abs{E})$ then $\abs{E}\in\{0,\tfrac12,1\}$ and if $\abs{E}=\tfrac12$ then $E$ is a half-cube.

For the proof, we will need the following simple lemma. This formula for the Jacobian of the restriction to a hyperplane is likely well known, but we could not find any reference, so we report it here.

\begin{lemma}\label{lem:jacobian-restriction}
    Fix $d \ge 2$. 
    Let $A\in GL(d, \R)$ be a linear transformation and let $\nu\in\R^d$ be a unit vector.
    The Jacobian determinant of the restriction of $A$ to the subspace orthogonal to $\nu$ is
    $\abs{\det(A)}\cdot \abs{(A^\intercal)^{-1}\nu}$.
\end{lemma}
\begin{proof}
    Take a Borel set $S\subseteq \nu^{\perp}$. By Fubini's Theorem, we have
    \begin{equation}\label{eq:tmp082}
        \abs*{S+\{t\nu: 0<t<1\}} = \Haus^{d-1}(S).
    \end{equation}
    Moreover, 
    \begin{equation}\label{eq:tmp083}
        \abs*{A(S+\{t\nu: 0<t<1\})} = \abs{\det(A)} \cdot \abs*{S+\{t\nu: 0<t<1\}} .
    \end{equation}
    Write $A\nu = u + \tilde u$, where $u$ is the orthogonal projection of $A\nu$ on the hyperspace $A(\nu^\perp)$. Notice that, for all $x\in\R^d$, $\langle (A^\intercal)^{-1}\nu, Ax\rangle = \langle \nu, x\rangle$, so $(A^\intercal)^{-1}\nu$ is orthogonal to the hyperspace $A(\nu^\perp)$. In particular, $\tilde u$ is a multiple of $(A^\intercal)^{-1}\nu$ . Hence, we have 
    \begin{equation}\label{eq:tmp084}
        \abs{\tilde u} = \frac{\abs{\langle A\nu, (A^\intercal)^{-1}\nu\rangle}}{\abs{(A^\intercal)^{-1}\nu}} = \frac1{\abs{(A^\intercal)^{-1}\nu}}.
    \end{equation}
    Thanks to Fubini's Theorem, we get
    \begin{equation}\label{eq:tmp085}
        \abs*{A(S+\{t\nu: 0<t<1\})}
        =
        \abs*{A(S)+\{tu+t\tilde u: 0<t<1\})}
        = \Haus^{d-1}(A(S))\abs{\tilde u}
    \end{equation}
    Combining \cref{eq:tmp082,eq:tmp083,eq:tmp084,eq:tmp085}, we obtain $\Haus^{d-1}(A(S)) = \abs{\det(A)}\cdot\abs{(A^\intercal)^{-1}\nu}\Haus^{d-1}(S)$ which is equivalent to the desired statement.
\end{proof}

\begin{proposition}[Rigidity for the Gaussian isoperimetric inequality in the cube]\label{prop:gaussian-cube-rigidity}
    For $E\subseteq(0,1)^d$ a set of finite perimeter, it holds that
    \begin{equation*}
        \Per(E, (0, 1)^d) \ge \sqrt{2\pi}I_\gamma(\abs{E}).
    \end{equation*}
    This inequality is an equality if and only if $E=\emptyset$ or $E=(0,1)^d$ or $E=\{x\in(0,1)^d:\, \hat v\cdot x \le \tfrac12\}$ for some $\hat v\in\{\pm e_1, \pm e_2, \dots, \pm e_d\}$.
\end{proposition}
\begin{proof}
    We follow the proof of \cite[Theorem 7]{Ros2005}.

    Let $\Phi_d:\R^d\to (0,1)^d$ be the map (see \cref{subsec:gaussian-isop} for the definition of $\Phi$)
    \begin{equation*}
        \Phi_d(x_1, x_2, \dots, x_d) \defeq (\Phi(x_1), \Phi(x_2), \dots, \Phi(x_d)).
    \end{equation*}
    Notice that $\varphi_d$ is the density of the Gaussian measure on $\R^d$ and that $\Phi_d$ is a diffeomorphism such that $(\Phi_d)_*(\gamma_d) = \restricts{\Leb^d}{(0,1)^d}$, in particular the Jacobian of $\Phi_d$ satisfies $\abs{\det D\Phi_d} = \varphi_d(x)$.

    For a finite perimeter set $E\subseteq (0, 1)^d$, the area formula \cite[Theorem 3.2.3]{Federer1969} combined with \cref{lem:jacobian-restriction} tells us that
    \begin{equation}\label{eq:tmp55}
        \Per(E, (0, 1)^d) = \int_{\partial^*(\Phi_d^{-1}(E))} 
        \abs{\det D\Phi_d} \cdot
        \abs{\big((D\Phi_d)^\intercal\big)^{-1}\nu} d\Haus^{d-1},
    \end{equation}
    where $\nu$ denotes the normal to the reduced boundary $\partial^*(\Phi_d^{-1}(E))$. We have that
    \begin{equation}\label{eq:tmp56}
        \abs{\big((D\Phi_d)^\intercal\big)^{-1}\nu} =
        \sqrt{2\pi}\sqrt{\sum_{i=1}^d \nu_i^2 e^{x_i^2}} \ge \sqrt{2\pi}
    \end{equation}
    and such inequality holds as an equality if and only if, for all $i= 1, 2,\dots, d$, we have $\nu_ix_i=0$. 
    Combining \cref{eq:tmp55,eq:tmp56}, we obtain
    \begin{equation}\label{eq:tmp57}
        \frac1{\sqrt{2\pi}} \Per(E, (0, 1)^d) \ge \int_{\partial^*(\Phi_d^{-1}(E))} \varphi_d \,d\Haus^{d-1} = \Per_{\gamma_d}(\Phi_d^{-1}(E)),
    \end{equation}
    and equality holds if and only if for $\Haus^{d-1}$-almost every point $x\in \partial^*(\Phi_d^{-1}(E))$, we have $\nu_i(x)x_i = 0$ for all $i=1,2,\dots, d$.
    The Gaussian isoperimetric inequality (see \cref{subsec:gaussian-isop}) tells us that
    \begin{equation}\label{eq:tmp58}
        \Per_{\gamma_d}(\Phi_d^{-1}(E)) \ge I_\gamma\big(\gamma_d(\Phi_d^{-1}(E))\big)
        = I_\gamma(\abs{E}),
    \end{equation}
    and equality holds if and only if $\Phi_d^{-1}(E)$ is an affine half-space (in particular, the normal $\nu$ to its boundary is constant) or it is the empty set or it is the whole $\R^d$.
    
    Combining \cref{eq:tmp57,eq:tmp58} we obtain the desired inequality. If the inequality of the statement is an equality, then in particular both \cref{eq:tmp57} and \cref{eq:tmp58} must be equalities. So, either $E=\emptyset$ or $E=(0,1)^d$ or $\Phi_d^{-1}(E)$ is an affine half-space and the normal $\nu$ to its boundary satisfies $\nu_i x_i=0$ for all $x\in\partial(\Phi_d^{-1}(E))$ and all $i=1,2,\dots, d$. In the latter case, take $1\le j\le d$ such that $\nu_j\not=0$. Then $x_j=0$ for all $x\in \partial(\Phi_d^{-1}(E))$ and thus $\Phi_d^{-1}(E) = \{x\in\R^d:\, x\cdot \hat v \le 0\}$ with $\hat v=e_j$ or $\hat v=-e_j$. The sought characterization for $E$ follows. 
\end{proof}

\section{Proof of \texorpdfstring{\cref{thm:main-halfmeasure}}{Theorem 1.1}}\label{sec:proof-halfmeasure}
The proof of \cref{thm:main-halfmeasure} is based on a compactness argument.
The main idea is that a sequence of perimeter-minimizing sets $E_n\subseteq (0,1)^d$ with $\abs{E_n}\to\frac12$ converges in a very strong sense to a minimizer with measure $\frac12$, and such minimizer must be a half-cube thanks to \cref{prop:gaussian-cube-rigidity}.

\begin{proof}[Proof of \cref{thm:main-halfmeasure}]
    Let $(E_n)_{n\in\N}\subseteq (0, 1)^d$ be a sequence of sets of finite perimeter such that
    \begin{itemize}
        \item $\abs{E_n} \to \frac12$ as $n\to\infty$.
        \item The set $E_n$ minimizes $\Per(E_n, (0,1)^d)$ among the sets with measure equal to $\abs{E_n}$.
    \end{itemize}
    The existence of $E_n$ is standard \cite[Proposition 12.30]{Maggi2012}. By \cite[Theorem 17.20]{Maggi2012}, we know that its reduced boundary $\partial^* E_n\cap (0,1)^d$ is a free-boundary integral rectifiable varifold in $(0,1)^d$ with constant mean curvature.

    We will prove that, for $n$ sufficiently large, $E_n$ coincides (up to negligible sets) with $\{x\in (0, 1)^d: x \cdot \hat v \le \abs{E_n}\}$ for some $\hat v \in \{\pm e_1, \pm e_2, \dots, \pm e_d\}$. The desired statement follows immediately.

    Let us show that the mean curvature of $E_n$ goes to $0$ as $n\to\infty$.
    The isoperimetric profile $I_{(0,1)^d}$ is concave \cite[Corollary 6.11]{Milman2009} and satisfies (see \cref{prop:gaussian-cube-rigidity})
    \begin{equation*}
        \sqrt{2\pi}I_\gamma(\lambda) \le I_{(0,1)^d}(\lambda) \le 1
    \end{equation*}
    for all $0\le \lambda\le 1$.
    Notice that the lower bound and the upper bound for $I_{(0,1)^d}$ are both smooth concave functions, they have the same value at $\lambda=\tfrac12$, and the derivatives at $\lambda=\tfrac12$ are equal to $0$. 
    Since $I_{(0,1)^d}$ is trapped between two such functions, it follows that $\partial I_{(0,1)^d}(\lambda)\to \{0\}$ as $\lambda\to \tfrac12$, where $\partial f$ denotes the superdifferential\footnote{The superdifferential $\partial f(\lambda)$ of a concave function $f$ at a point $\lambda$ is the set of slopes $v\in\R$ so that $f(\lambda+t)\le f(\lambda) + vt$ for all $t\in\R$ so that $\lambda+t$ belongs to the domain of $f$.} of the concave function $f$.    
    Since $E_n$ is a minimizer for the relative isoperimetric inequality, its mean curvature belongs to $\partial I_{(0,1)^d}(\abs{E_n})$ as proven in \cite[Proposition 4.8]{RitoreVernadakis2015} (see also \cite[Corollary 2.9.]{SternbergZumbrun1999} for the case of ambient spaces with smooth boundary) and therefore we deduce that the mean curvature of $E_n$ goes to $0$ as $n\to\infty$.

    By compactness~\cite[Theorem 12.26]{Maggi2012}, up to taking a subsequence, we may assume that $E_n$ converges to $E_\infty$ in the sense that ${\mathds 1}_{E_n}\to {\mathds 1}_{E_\infty}$ in $L^1$ and $D{\mathds 1}_{E_n}\overset{\ast}{\rightharpoonup} D{\mathds 1}_{E_\infty}$ in the open set $(0,1)^d$.
    Notice that $\Per(E_n, (0,1)^d)\le 1$ and therefore, by lower semicontinuity of the perimeter, we have $\Per(E_\infty, (0, 1)^d) \le 1$. Moreover $\abs{E_\infty} = \lim \abs{E_n} = \frac12$. 
    Thus, by \cref{prop:gaussian-cube-rigidity}, we obtain that, without loss of generality, $E_\infty=\{x\in (0,1)^d:\, x_d \le \frac12\}$.
    In particular, we have $\Per(E_n, (0, 1)^d) \to \Per(E_\infty, (0, 1)^d)$ and thus, applying \cite[Theorem 6.4]{Allard1972}, the boundaries $\partial^*E_n$ converge to $\partial^*E_\infty$ in the varifold sense.
    
    We have verified all the assumptions necessary to apply \cite[Regularity Theorem]{Allard1972} in the interior and \cite[Theorem 1.1]{EdelenLi2022} at the boundary, thus we have that, for $n$ sufficiently large, $\partial^* E_n\cap(0,1)^d$ is a graph over $\partial^* E_\infty\cap(0,1)^d = \{x\in (0,1)^d:\, x_d = \frac12\}$.
    Since $\partial^*E_\infty\cap(0,1)^d$ is flat, it follows in particular that 
    \begin{equation*}
        \Per(E_n, (0,1)^d)
        =
        \Haus^{d-1}(\partial^*E_n\cap (0,1)^d)
        \ge
        \Haus^{d-1}(\partial^*E_\infty\cap (0,1)^d)
        =1 ,
    \end{equation*}
    with equality if and only if $\partial^*E_n$ is the graph of a \emph{constant} function over $\partial^* E_\infty$, which is exactly the desired statement.
\end{proof}

\section{Proof of \texorpdfstring{\cref{thm:main-highdim}}{Theorem 1.2}}\label{sec:proof-highdim}

We will need two simple technical lemmas. 
Let us emphasize that the theme of this whole section is obtaining estimates that do not depend on the dimension $d$.

\begin{lemma}\label{lem:gaussian-boundary-stability}
    Let $F\subseteq\R^d$ be a set of locally finite perimeter and let $H\subseteq\R^d$ be an affine half-space. Let $\ell\defeq \dist(0_{\R^d}, \partial H)$ (observe that $\Haus^{d-1}_{\gamma_d}(\partial H)=\varphi(\ell)$) and let $\pi_{\partial H}:\R^d\to\partial H$ be the projection on the hyperplane $\partial H$.
    For any positive real number $r>0$, we have
    \begin{equation*}
        \Haus^{d-1}_{\gamma_d}\Big(\pi_{\partial H}\big(
            \partial^*F\cap \{x:\,\dist(x, \partial H) < r\}
        \big)\Big) 
        \ge 
        \varphi(\ell) - \frac{\varphi(\ell)}{\varphi(\ell+r)} 
        \frac{\gamma_d(F\triangle H)}{r} .
    \end{equation*}
\end{lemma}
\begin{proof}
    Let $V\defeq\pi_{\partial H}\big(\partial^*F\cap \{x:\,\dist(x, \partial H) < r\}\big)$.
    Let $\nu_{\partial H}$ be the normal to $\partial H$, oriented so that $\ell\nu_{\partial H}\in \partial H$.
    For each $x\in \partial H$, let us consider the $1$-dimensional slice $F_x \defeq F\cap \{x+t\nu_{\partial H}: t\in\R\}$; define $H_x$ analogously. For $\Haus^{d-1}$-almost every $x\in\partial H$, the set $F_x$ is (locally) made of finitely many disjoint intervals and all the extreme points of such intervals belong to $\partial^* F$ (see \cite[Remark 18.13]{Maggi2012}).

    For $\Haus^{d-1}$-almost every $x\in\partial H\setminus V$, we have\footnote{With $x+\nu_{\partial H}\,(-r,r)$ we denote the set $\{x+t\nu_{\partial H}:\, t\in(-r,r)\}$.} that $F_x\cap (x+\nu_{\partial H}\,(-r,r))$ is either empty or equal to $x+\nu_{\partial H}(-r,r)$ (up to negligible sets) because $F_x$ cannot have any boundary point in the interval $x+\nu_{\partial H}(-r,r)$. In both cases, $\Haus^{1}_{\gamma_d}(F_x\triangle H_x) \ge r\varphi_d(x+r\nu_{\partial H})$.

    Observe that $\varphi_d(x+r\nu_{\partial H})=\frac{\varphi(\ell+r)}{\varphi(\ell)}\varphi_d(x)$.
    By Fubini's Theorem, we get
    \begin{align*}
        \gamma_d(F\triangle H) 
        &\ge
        \int_{\partial H\setminus V} \Haus^1_{\gamma_d}(F_x\triangle H_x) \,d\Haus^{d-1}(x)
        \\
        &\ge
        r 
        \int_{\partial H\setminus V} 
        \varphi_d(x+r\nu_{\partial H})
        \,d\Haus^{d-1}(x)
        =
        r\frac{\varphi(\ell+r)}{\varphi(\ell)}
        \Haus^{d-1}_{\gamma_d}(\partial H\setminus V)
    \end{align*}
    and the desired statement follows.
\end{proof}

We will apply the following lemma only with the function $f(t)=\exp(\frac12t^2)$ and it is possible to prove a sharper result in this case, but we decided to prioritize clarity.
Informally, the following lemma is a quantitative way to state the fact that a hyperplane with distance $\ell$ from the origin cannot be a subset of the strip $\{x\in\R^d:\, \abs{x_i} < \tfrac\ell2\}$.

\begin{lemma}\label{lem:lowerbound-integral-generic}
    For any $\ell>0$ there is a constant $c=c(\ell)>0$ such that the following statement holds.

    Let $\Sigma\subseteq\R^d$ be an affine hyperplane with $\dist(0_{\R^d}, \Sigma)=\ell$ (which is equivalent to $\Haus^{d-1}_{\gamma_d}(\Sigma) = \varphi(\ell)$). For any Borel subset $V\subseteq \Sigma$ and any nondecreasing function $f:[0,\infty)\to [0,\infty)$, we have
    \begin{equation*}
        \int_{V} f(\abs{x_i})\,d\Haus^{d-1}_{\gamma_d}(x) 
        \ge 
        \big(c - \Haus^{d-1}_{\gamma_d}(\Sigma\setminus V)\big)f(\tfrac{\ell}{2})
    \end{equation*}
    for any $i\in\{1, 2,\dots, d\}$.
\end{lemma}
\begin{proof}
    Since $f$ is nondecreasing, we have
    \begin{equation*}
    \begin{aligned}
        \int_{V} f(\abs{x_i})\,d\Haus^{d-1}_{\gamma_d}(x)
        &\ge
        \Haus^{d-1}_{\gamma_d}(V\cap \{x:\, \abs{x_i} \ge \tfrac\ell2\}) f(\tfrac\ell2)
        \\
        &\ge
        \big[
            \varphi(\ell) - 
            \Haus^{d-1}_{\gamma_d}(\Sigma\cap \{x:\, \abs{x_i} < \tfrac\ell2\})
            -
            \Haus^{d-1}_{\gamma_d}(\Sigma\setminus V)\big]
        f(\tfrac\ell2) ,
    \end{aligned}
    \end{equation*}
    therefore to prove the statement it is sufficient to show that there exists a constant $c=c(\ell)$ so that
    \begin{equation*}
        \Haus^{d-1}_{\gamma_d}(\Sigma\cap \{x:\, \abs{x_i} < \tfrac\ell2\})
        \le \varphi(\ell)-c.
    \end{equation*}
    Let $\sigma_1, \sigma_2,\dots, \sigma_d\in\R^d$ be an orthonormal basis such that $\Sigma = \ell\sigma_d + \langle \sigma_1, \dots, \sigma_{d-1}\rangle$. In particular, $\sigma_d$ is the unit normal to $\Sigma$.
    Up to rotation, we may assume that $e_i\in\langle \sigma_1, \sigma_d\rangle$, where $e_i$ is the $i$-th element of the standard basis of $\R^d$ (so that $e_i\cdot x=x_i$). Then, it must be $e_i = (\sigma_d)_i\sigma_d + \sqrt{1-(\sigma_d)_i^2}\sigma_1$.
    Using the parametrization of $\Sigma$ given by $\R^{d-1}\ni y \mapsto y_1\sigma_1 + \cdots + y_{d-1}\sigma_{d-1} + \ell\sigma_d$, we get
    \begin{align*}
        \Haus^{d-1}_{\gamma_d}(\Sigma\cap \{x:\, \abs{x_i} < \tfrac\ell2\})
        &=
        \varphi(\ell)
        \gamma_{d-1}\Big(
            \big\{y\in\R^{d-1}:\, 
                \abs{(\sigma_d)_i\ell + \sqrt{1-(\sigma_d)_i^2}y_1} < \tfrac\ell2
            \big\}\Big)
        \\
        &=
        \varphi(\ell)
        \gamma_{1}\Big(
            \big\{t\in\R:\, 
                \abs{(\sigma_d)_i\ell + \sqrt{1-(\sigma_d)_i^2}t} < \tfrac\ell2
            \big\}\Big) .
    \end{align*}
    Let $q\defeq (\sigma_d)_i$. Consider the set appearing at the right-hand side of the last equation. It is a (possibly empty) interval with length $\frac{\ell}{\sqrt{1-q^2}}$ (empty if $\abs{q}=1$) and it contains $0$ if and only if $\abs{q} < \frac12$. So, either it does not contain zero and thus its Gaussian measure is less than $\frac12$ or its length is bounded by $\frac{2\ell}{\sqrt{3}}$ and thus its Gaussian measure is bounded by $1-c(\ell)$ for a suitable constant $c(\ell)>0$. In both cases, the desired statement follows.
\end{proof}

We are now ready to prove \cref{thm:main-highdim}. The proof \emph{lives} in the Gaussian space $(\R^d, \gamma_d)$ instead of the cube $(0,1)^d$ (we move everything there with the same map that appeared in the proof of \cref{prop:gaussian-cube-rigidity}).
Oversimplifying, the idea of the proof is that if the statement were false, then we would find a set $F\subseteq\R^d$ that is \emph{almost} a half-space and minimizes a boundary integral (see \cref{eq:tmp88}) that is not minimized by half-spaces; this yields a contradiction.

\begin{proof}[Proof of \cref{thm:main-highdim}]
    The inequality $I_{(0,1)^{d+1}}\le I_{(0,1)^d}$ follows from the fact that the map $E\mapsto E\times (0,1)$ transforms a subset of $(0,1)^d$ into a subset of $(0,1)^{d+1}$ with the same perimeter and the same measure. The concavity of $I_{(0,1)^\infty}$ follows from the concavity of $I_{(0,1)^d}$~\cite[Corollary 6.11]{Milman2009}.

    For the second part of the statement, fix a dimension $d\ge 1$ and $0<\lambda<1$ different from $\tfrac12$. Let $E\subseteq (0,1)^d$ be a set of finite perimeter such that $\abs{E}=\lambda$ and $\Per(E, (0,1)^d)=I_{(0,1)^d}(\lambda)$ (a set $E$ with these properties exists thanks to \cite[Proposition 12.30]{Maggi2012}). By repeating the proof of \cref{prop:gaussian-cube-rigidity} for the set $E$, we obtain that
    \begin{equation}\label{eq:tmp88}
        \frac1{\sqrt{2\pi}}I_{(0,1)^d}(\lambda) - I_\gamma(\lambda) 
        \ge 
        \Per_{\gamma_d}(F) - I_\gamma(\lambda) +
        \int_{\partial^* F}
        \sqrt{\sum_{i=1}^d (\nu_{\partial^*F})_i^2 e^{x_i^2}}-1 \,d\Haus^{d-1}_{\gamma_d}(x),
    \end{equation}
    where $F=\Phi_d^{-1}(E)$ and $\nu_{\partial^*F}$ is the normal to $\partial^* F$. Notice that $\gamma_d(F)=\lambda$. The intuitive idea is that the two terms at the right-hand side cannot be simultaneously small: $\Per_{\gamma_d}(F) - I_\gamma(\lambda)$ is small if $\partial^* F$ is close to an affine hyperplane, while the integral is strictly positive if $\partial^* F$ is an affine half-space not containing the origin (which is guaranteed by the condition $\abs{E}\not=\frac12$).
    There is a crucial difficulty: all our estimates must be uniform in the dimension $d$, because we want to show that $\frac1{\sqrt{2\pi}}I_{(0,1)^d}(\lambda) - I_\gamma(\lambda)$ is bounded away from $0$ for all $d\ge 1$.

    Let $\delta\defeq \Per_{\gamma_d}(F)-I_\gamma(\lambda)$.
    We are going to prove that there exist two constants $\delta_1(\lambda)$ and $c_1(\lambda)$ (independent of $d$) so that 
    \begin{equation}\label{eq:subtask}
        \text{If $\delta<\delta_1$ then } \int_{\partial^* F}\sqrt{\sum_{i=1}^d (\nu_{\partial^*F})_i^2 e^{x_i^2}}-1 \,d\Haus^{d-1}_{\gamma_d}(x) \ge c_1.    
    \end{equation}
    Assuming that \cref{eq:subtask} holds, thanks to \cref{eq:tmp88}, we deduce $\frac1{\sqrt{2\pi}}I_{(0,1)^d}(\lambda) - I_\gamma(\lambda) \ge \min\{\delta_1, c_1\}$ and therefore the statement of the Theorem follows.

    Let us now prove \cref{eq:subtask}. As will be clear from the proof, one can choose $\delta_1\defeq \min\{1, \big(\tfrac{\abs{\Phi^{-1}(\lambda)}}4 \big)^4\}$ (notice that $\Phi^{-1}(\lambda)\not=0$ because $\lambda\not=\frac12$).
    By the dimension free stability of the Gaussian isoperimetric inequality \cite[Main Theorem and Proposition 4]{BarchiesiBrancoliniJulin2017}, there is an affine half-space $H\subseteq\R^d$ with $\gamma_d(H)=\gamma_d(F)=\lambda$ so that
    \begin{equation}\label{eq:tmp83}
        \gamma_d(F\triangle H) \le C_1\delta^{\frac12},
    \end{equation}
    for a constant $C_1=C_1(\lambda)>0$ independent of the dimension $d$. Let $\nu_{\partial H}\in\R^d$ be the outer normal to the affine hyperplane $\partial H$. Let $\ell>0$ denote the distance between $\partial H$ and $0_{\R^d}$. Notice that $I_\gamma(\lambda)=\Haus^{d-1}_{\gamma_d}(\partial H) = \varphi(\ell)$ (therefore if a constant depends only on $\ell$ then it depends only on $\lambda$). 
    We also know \cite[Corollary 2 and Proposition 4]{BarchiesiBrancoliniJulin2017}
    \begin{equation}\label{eq:tmp84} 
        \int_{\partial^*F} \abs{\nu_{\partial^*F} - \nu_{\partial H}}^2 \,d\Haus^{d-1}_{\gamma_d} \le C_2\delta,
    \end{equation}
    for a constant $C_2=C_2(\lambda)>0$ independent of the dimension $d$.

    By Jensen's inequality applied in the integrand (on the concave function $\sqrt{\emptyparam}-1$; observe that $\sum_{i=1}^d (\nu_{\partial^*F})_i^2=1$), we have
    \begin{align*}
        &\int_{\partial^* F}
        \sqrt{\sum_{i=1}^d (\nu_{\partial^*F})_i^2 e^{x_i^2}}-1 \,d\Haus^{d-1}_{\gamma_d}(x)
        \\
        &\quad\ge
        \sum_{i=1}^d
        \int_{\partial^* F} (\nu_{\partial^*F})_i^2
        \Big[
        \exp\big(\tfrac12 x_i^2\big) - 1
        \Big] \,d\Haus^{d-1}_{\gamma_d}(x).
    \end{align*}
    Define $U_i\defeq \{x\in \partial^*F:\, \nu_{\partial^*F}(x)_i^2 \ge \frac12 (\nu_{\partial H})_i^2\}$. Using the newly-defined sets $U_i$, we can continue the chain of inequalities
    \begin{equation}\label{eq:tmp21}
    \begin{aligned}
        \ge \frac12 \sum_{i=1}^d (\nu_{\partial H})_i^2
        \int_{U_i}\Big[
        \exp\big(\tfrac12 x_i^2\big) - 1
        \Big] \,d\Haus^{d-1}_{\gamma_d}(x) .
    \end{aligned}
    \end{equation}
    We show that for many indexes $i$, the set $U_i$ almost saturates $\partial^*F$.
    We have
    \begin{equation*}
    \begin{aligned}
        \frac12 \sum_{i=1}^d (\nu_{\partial H})_i^2 \Haus^{d-1}_{\gamma_d}(\partial^*F\setminus U_i)
        &\le
        \sum_{i=1}^d \int_{\partial^*F\setminus U_i} \abs{\nu_{\partial^*F}(x)_i^2-(\nu_{\partial H})_i^2} \,d\Haus^{d-1}_{\gamma_d}(x)
        \\
        &\le
        \int_{\partial^*F} \sum_{i=1}^d \abs{\nu_{\partial^*F}(x)_i^2-(\nu_{\partial H})_i^2} \,d\Haus^{d-1}_{\gamma_d}(x).
    \end{aligned}
    \end{equation*}
    For any two vectors $u, v\in\R^d$, the Cauchy–Schwarz inequality implies $\sum_{i=1}^d\abs{u_i^2-v_i^2} \le \abs{u-v}\abs{u+v}$, and thus we can continue the chain of inequalities as follows
    \begin{equation*}
    \begin{aligned}
        &\le
        2\int_{\partial^*F} \abs{\nu_{\partial^*F}(x)-\nu_{\partial H}}
        \,d \Haus^{d-1}_{\gamma_d}(x)
        \\
        &\le 
        2\Big(
            \int_{\partial^*F} \abs{\nu_{\partial^*F}(x)-\nu_{\partial H}}^2
            \,d \Haus^{d-1}_{\gamma_d}(x)
        \Big)^{\frac12} 
        \Per_{\gamma_d}(F)^{\frac12}
        \le
        2 \sqrt{C_2\delta} \Per_{\gamma_d}(F)^{\frac12}
    \end{aligned}
    \end{equation*}
    where in the last step we have applied \cref{eq:tmp84}. Hence (assuming $\delta\le 1$, so that $\Per_{\gamma_d}(F)$ is controlled) we deduce
    \begin{equation*}
        \sum_{i=1}^d (\nu_{\partial H})_i^2 \Haus^{d-1}_{\gamma_d}(\partial^*F\setminus U_i)
        \le C_3 \delta^{\frac12},
    \end{equation*}
    where $C_3=C_3(\lambda)>0$ is a constant. Hence, there is a subset $J\subseteq \{1, 2,\dots, d\}$ such that
    \begin{equation*}
        \sum_{i\in J} (\nu_{\partial H})_i^2 \ge \frac12 \quad \text{ and } \quad
        \Haus^{d-1}_{\gamma_d}(\partial^*F\setminus U_i) \le 2C_3\delta^{\frac12} 
        \text{ for all $i\in J$.}
    \end{equation*}

    Now, fix $i\in J$. We show a lower bound for $\int_{U_i} \exp(\tfrac12 x_i^2)-1\,d\Haus^{d-1}_{\gamma_d}(x)$ by projecting onto $\partial H$.
    Let $\pi_{\partial H}:\R^d\to\partial H$ be the orthogonal projection on the affine hyperplane $\partial H$. Denote with $B(\partial H, r)$ the set of points with distance $< r$ from $\partial H$, i.e., $B(\partial H, r) \defeq \{x\in\R^d:\, \dist(x, \partial H) < r\}$. Observe that for any subset $U\subseteq B(\partial H, r)$, since $\pi_{\partial H}$ is $1$-Lipschitz, the area formula~\cite[Theorem 3.2.3]{Federer1969} gives\footnote{To justify the second step notice that $y\in \pi_{\partial H}^{-1}(x)\cap B(\partial H, r)$ implies $y=x+t\nu_{\partial H}$ for some $t\in[-r,r]$ and thus $\varphi_d(y)=\varphi_d(x+t\nu_{\partial H})=\frac{\varphi(\ell+t)}{\varphi(\ell)}\varphi_d(x)\ge \frac{\varphi(\ell+r)}{\varphi(\ell)}\varphi_d(x)$.}
    \begin{align*}
        \Haus^{d-1}_{\gamma_d}(U) 
        &\ge
        \int_{\partial H} 
            \sum_{y\in \pi_{\partial H}^{-1}(x)\cap U} \varphi_d(y)
        \,d\Haus^{d-1}(x)
        \ge
         \int_{\partial H} 
            \sum_{y\in \pi_{\partial H}^{-1}(x)\cap U} \frac{\varphi(\ell+r)}{\varphi(\ell)}
        \,d\Haus^{d-1}_{\gamma_d}(x)
        \\
        &\ge
        \frac{\varphi(\ell+r)}{\varphi(\ell)} \Haus^{d-1}_{\gamma_d}(\pi_{\partial H}(U)),
    \end{align*}
    hence
    \begin{equation}\label{eq:tmp93}
        \Haus^{d-1}_{\gamma_d}\big(\pi_{\partial H}(U)\big) \le
        \frac{\varphi(\ell)}{\varphi(\ell+r)}\Haus^{d-1}_{\gamma_d}(U).
    \end{equation}
    Define the subset $V_i\subseteq\partial H$ as 
    \begin{equation*}
        V_i\defeq \pi_{\partial H}\Big(U_i\cap B(\partial H, \delta^{\frac14})\Big) .
    \end{equation*}
    Let us show that $V_i$ \emph{saturates} $\partial H$, 
    \begin{equation}\begin{aligned}
        \Haus^{d-1}_{\gamma_d}(V_i)
        &\ge
        \Haus^{d-1}_{\gamma_d}\Big(
            \pi_{\partial H}\big(
                \partial^* F\cap B(\partial H, \delta^{\frac14})
            \big)
        \Big)
        -
        \Haus^{d-1}_{\gamma_d}\Big(
            \pi_{\partial H}\big(
                (\partial^*F\setminus U_i)\cap
                B(\partial H, \delta^{\frac14})
            \big)
        \Big)
        \\
        &\ge
        \varphi(\ell) - 
        \frac{\varphi(\ell)}{\varphi(\ell+\delta^{\frac14})}
        \frac{\gamma_d(F\triangle H)}{\delta^{\frac14}}
        -\frac{\varphi(\ell)}{\varphi(\ell+\delta^{\frac14})}
        \Haus^{d-1}_{\gamma_d}\Big(\partial^*F\setminus U_i\Big) ,
    \end{aligned}\end{equation}
    where we have used \cref{lem:gaussian-boundary-stability,eq:tmp93}. Combining the latter inequality with \cref{eq:tmp83} and with the estimate $\Haus^{d-1}_{\gamma_d}(\partial^*F\setminus U_i) \le 2C_3\delta^{\frac12}$, if $\delta$ is sufficiently small with respect to $\ell$ (which depends only on $\lambda$), we obtain
    \begin{equation}\begin{aligned}
        \Haus^{d-1}_{\gamma_d}(V_i)
        \ge
        \varphi(\ell) - 
        C_4\delta^{\frac14} ,
    \end{aligned}\end{equation}
    for a constant $C_4=C_4(\lambda)>0$.

    For a real parameter $0<\kappa$, let $L_\kappa:\R\to\R$ be the function
    \begin{equation*}
        L_{\kappa}(s) \defeq 
        \begin{cases}
            s + \kappa &\quad\text{if $s\le -\kappa$,} \\
            0 &\quad\text{if $-\kappa \le s\le \kappa$,} \\
            s - \kappa &\quad\text{if $\kappa \le s$.}
        \end{cases}
    \end{equation*}
    Observe that, if $x\in \partial H$ and $y = x + t\nu_{\partial H}$, then 
    $\exp(\frac12 y_i^2)-1 \ge \exp(\frac12 L_{\nu_it}(x_i)^2)-1$.
    Therefore, by repeating the same argument we employed to establish \cref{eq:tmp93}, we get
    \begin{equation*}
        \int_{U_i}\exp\big(\tfrac12 x_i^2\big)-1 \,d\Haus^{d-1}_{\gamma_d}(x)
        \ge 
        \frac{\varphi(\ell+\delta^{\frac14})}{\varphi(\ell)}
        \int_{V_i} 
        \exp\big(\tfrac12 L_{\nu_i\delta^{\frac14}}(x_i)^2\big) - 1
        \,d\Haus^{d-1}_{\gamma_d}(x).
    \end{equation*}
    Applying \cref{lem:lowerbound-integral-generic}, if we assume that $\nu_i\delta^{\frac14}<\ell/4$, the last estimate implies
    \begin{equation*}
        \int_{U_i}\exp\big(\tfrac12 x_i^2\big)-1 \,d\Haus^{d-1}_{\gamma_d}(x)
        \ge
        c_5,
    \end{equation*}
    for some constant $c_5=c_5(\ell)>0$ which does not depend on the dimension $d$.
    Combining the latter inequality with \cref{eq:tmp21}, we obtain
    \begin{align*}
        \int_{\partial^* F}
        \sqrt{\sum_{i=1}^d (\nu_{\partial^*F})_i^2 e^{x_i^2}}-1 \,d\Haus^{d-1}_{\gamma_d}(x)
        &\ge \frac12 \sum_{i\in J} (\nu_{\partial H})_i^2
        \int_{U_i}\Big[
        \exp\big(\tfrac12 x_i^2\big) - 1
        \Big] \,d\Haus^{d-1}_{\gamma_d}(x)
        \\
        &\ge \frac12 c_5 \sum_{i\in J} (\nu_{\partial H})_i^2
        \ge \frac14 c_5,
    \end{align*}
    that concludes the proof of \cref{eq:subtask} as $c_5$ is a positive constant that does not depend on the dimension $d\ge 1$.
\end{proof}

\printbibliography

\end{document}

%% file: settings.tex
\calclayout

\usepackage[utf8]{inputenc}
\usepackage{dsfont} 
\usepackage{comment}

\usepackage{hyperref}
\hypersetup{
	colorlinks = true,
	linkcolor = {blue},
	urlcolor = {black},
	citecolor = {black}
}

\usepackage{amssymb}
\usepackage{enumitem}
\usepackage[capitalize]{cleveref}
\usepackage{color}

\usepackage[backend=biber,
            style=alphabetic,
            bibencoding=utf8,
            language=auto,
            autolang=other,
            giveninits=true,
            doi=false,
            isbn=false,
            url=false,
            maxnames=10,
            sorting=nty]{biblatex}
\usepackage{asymptote}

\usepackage{mathtools}
\usepackage{faktor} 
\usepackage{tikz}
\usepackage{mathrsfs} 
\usetikzlibrary{babel} 
\usetikzlibrary{cd} 

\crefname{enumi}{}{} 

\newtheorem{theorem}{Theorem}[section]
\newtheorem*{theorem*}{Theorem}
\newtheorem{lemma}[theorem]{Lemma}
\newtheorem{proposition}[theorem]{Proposition}

\theoremstyle{remark}

\newtheorem*{remark*}{Remark}
\newtheorem{question}{Open question}[section]

\theoremstyle{definition}

\usepackage[backend=biber,style=alphabetic,bibencoding=utf8,language=auto,autolang=other,giveninits=true,doi=false,isbn=false,url=false,maxnames=10,sorting=nty]{biblatex}

\newcommand{\N}{\ensuremath{\mathbb N}} 
\newcommand{\R}{\ensuremath{\mathbb R}} 

\newcommand{\defeq}{\ensuremath{\coloneqq}}

\newcommand{\emptyparam}{\ensuremath{\,\cdot\,}}

\makeatletter 
\DeclarePairedDelimiter{\@tmpabs}{\lvert}{\rvert}
\newcommand{\@absstar}[1]{{\@tmpabs*{#1}}}
\newcommand{\@absnostar}[2][]{{\@tmpabs[#1]{#2}}}
\newcommand{\abs}{\@ifstar\@absstar\@absnostar}
\makeatother

\makeatletter 
\DeclarePairedDelimiter{\@tmpnorm}{\lVert}{\rVert}
\newcommand{\@normstar}[1]{{\@tmpnorm*{#1}}}
\newcommand{\@normnostar}[2][]{{\@tmpnorm[#1]{#2}}}
\newcommand{\norm}{\@ifstar\@normstar\@normnostar}
\makeatother

\newcommand\eps{\ensuremath{\varepsilon}}
\DeclareMathOperator{\Per}{Per}

\newcommand{\Haus}{\ensuremath{\mathscr H}} 
\newcommand{\Leb}{\ensuremath{\mathscr L}} 

\usepackage{pythonhighlight} 

\newcommand{\restricts}[2] {
	#1 
	\raisebox{-.3ex}{$|$}_{#2}
}

\newcommand{\prob}{\ensuremath{\mathcal{P}}}

\DeclareMathOperator{\dist}{dist} 

%% file: isop_profiles.tikz
\begin{tikzpicture}[domain=0:1,scale=5,
	samples at={0.0, 0.01, 0.02, 0.03, 0.04, 0.05, 0.06, 0.07, 0.08, 0.09, 0.10, 0.11, 0.12, 0.13, 0.14, 0.15, 0.16, 0.17, 0.18, 0.19, 0.2, 0.21, 0.22, 0.23, 0.24, 0.25, 0.26, 0.27, 0.28, 0.29, 0.3, 0.31, 0.32, 0.33, 0.34, 0.35, 0.36, 0.37, 0.38, 0.39, 0.4, 0.41, 0.42, 0.43, 0.44, 0.45, 0.46, 0.47, 0.48, 0.49, 0.5, 0.51, 0.52, 0.53, 0.54, 0.55, 0.56, 0.57, 0.58, 0.59, 0.6, 0.61, 0.62, 0.63, 0.64, 0.65, 0.66, 0.67, 0.68, 0.69, 0.7, 0.71, 0.72, 0.73, 0.74, 0.75, 0.76, 0.77, 0.78, 0.79, 0.8, 0.81, 0.82, 0.83, 0.84, 0.85, 0.86, 0.87, 0.88, 0.89, 0.9, 0.91, 0.92, 0.93, 0.94, 0.95, 0.96, 0.97, 0.98, 0.99, 1}]

\makeatletter
\pgfmathdeclarefunction{erf}{1}{%
  \begingroup
    \pgfmathparse{#1 > 0 ? 1 : -1}%
    \edef\sign{\pgfmathresult}%
    \pgfmathparse{abs(#1)}%
    \edef\x{\pgfmathresult}%
    \pgfmathparse{1/(1+0.3275911*\x)}%
    \edef\t{\pgfmathresult}%
    \pgfmathparse{%
      1 - (((((1.061405429*\t -1.453152027)*\t) + 1.421413741)*\t 
      -0.284496736)*\t + 0.254829592)*\t*exp(-(\x*\x))}%
    \edef\y{\pgfmathresult}%
    \pgfmathparse{(\sign)*\y}%
    \pgfmath@smuggleone\pgfmathresult%
  \endgroup
}

\draw[-] (0,0) -- (1,0);

\draw[color=red]    plot (\x,
	{min(1,   min(    sqrt(pi*\x),    sqrt(pi*(1-\x))   ))}
); 

\draw[color=blue,dashed]    plot (\x,
	{min(
		min(1,   min(    sqrt(pi*\x),  1.5*(1.333*pi)^0.333*\x^0.666   )),
		min(1,   min(    sqrt(pi*(1-\x)),  1.5*(1.333*pi)^0.333*(1-\x)^0.666   ))}
); 

\draw plot[smooth] file {I-gamma.table}; 

\draw node[below] at (0, 0) {$0$};
\draw node[below] at (0.5, 0) {$\lambda$};
\draw node[below] at (1, 0) {$1$};

\draw[color=red] node at (1, 1) {$I_{(0,1)^2}$};
\draw[color=blue] node at (1.2, 0.9) {$I_{(0,1)^3}$ (conjectural)};
\draw[color=black] node at (1, 0.8) {$\sqrt{2\pi}I_\gamma$};
 
\end{tikzpicture}